\documentclass{amsproc}

\usepackage{amsmath,amssymb,amsfonts}
\usepackage[dvipsnames,usenames]{color}
\usepackage{array}
\usepackage{accents}

\newcolumntype{L}{>{\displaystyle}l}
\newcolumntype{C}{>{\displaystyle}c}
\newcolumntype{R}{>{\displaystyle}r}

\newcommand{\R}{\mathbb R}
\newcommand{\N}{\mathbb N}

\newcommand{\C}{\mathbb C}

\newcommand{\hilb}{\mathcal H}

\renewcommand{\Im}{\mathrm{Im}}

\newcommand{\til}{\widetilde}

\newcommand{\I}{\mathfrak{i}}
\newcommand{\diff}{\mathrm{d}}
\renewcommand{\bar}{\overline}
\newcommand{\Dom}{\mathrm{Dom}}
\newcommand{\Leb}{\mathrm{L}}
\newcommand{\Id}{\mathrm{Id}}
\newcommand{\Ran}{\mathrm{Ran}}
\newcommand{\Mult}{\mathrm{Mult}}
\newcommand{\ntto}{\:\scriptsize{\xrightarrow{\vee\:}}\:}
\newcommand{\cut}{\C\setminus\R}
\newcommand{\cutN}{(\C\setminus\R)^n}

\newcommand{\relone}[2]{\left(\begin{array}{c}
#1 \\
#2
\end{array}\right)}
\newcommand{\relpair}[4]{\left\{\left(\begin{array}{c}
#1 \\
#2
\end{array}\right),\left(\begin{array}{c}
#3 \\
#4
\end{array}\right)\right\}}

\newcommand{\matfour}[4]{\left[\begin{array}{cc}
#1 & #2 \\
#3 & #4
\end{array}\right]}

\newcommand{\ie}{\textit{i.e.}\/ } 
\newcommand{\eg}{\textit{e.g.}\/ } 
\newcommand{\cf}{\textit{cf.}\/ } 

\usepackage{bm}
\renewcommand{\vec}{\bm}

\allowdisplaybreaks

\theoremstyle{definition}
\newtheorem{define}{Definition}[section]
\newtheorem{example}[define]{Example}
\newtheorem{remark}[define]{Remark}
\newtheorem{remarks}[define]{Remarks}

\theoremstyle{plain}
\newtheorem{lemma}[define]{Lemma}
\newtheorem{thm}[define]{Theorem}
\newtheorem{prop}[define]{Proposition}

\numberwithin{equation}{section}

\begin{document}

\title[Operator representations]{Operator representations of Herglotz-Nevanlinna functions in several variables}

\author{Erik Avelin}
\address{Erik Avelin, Department of Mathematics, Stockholm University, SE-106 91 Stockholm, Sweden.}
\curraddr{}
\email{erik.avelin@math.su.se}

\author{Mitja Nedic}
\address{Mitja Nedic, Department of Mathematics, Physics and Chemical Engineering, School of Engineering, Jönköping University, SE-551 11 Jönköping, Sweden, orc-id: 0000-0001-7867-5874}
\curraddr{}
\email{mitja.nedic@ju.se}



\begin{abstract}
In this paper, we present an operator representation for Herglotz-Nevanlinna functions in several variables using self-adjoint linear relations. Moreover, a description of the reproducing kernel Hilbert space corresponding to a Herglotz-Nevanlinna function of several variables is also presented, along with a comparison to some existing result concerning the class of Loewner functions.
\end{abstract}

\maketitle

\section{Introduction}

In the poly-upper half-plane $\C^{+n} := \{\vec{z} \in \C^n~|~\Im[z_j] > 0 \;\text{for}\;j=1,\ldots,n\}$, we consider holomorphic functions with non-negative imaginary part. These functions are called \emph{Herglotz-Nevanlinna functions} and some common points of interest in their study are various representation results for these functions. For example, an operator representation may entail being able to write the value of the function at a certain point with the help of an expression involving \eg the resolvent of a self-adjoint operator on some Hilbert space, \cf \cite{AglerEtal2012,AglerEtal2016,KreinLanger1977,LangerTextorius1977}. Other variants of representations can be \eg integral representations, \cf \cite{Cauer1932,KacKrein1974,LugerNedic2017,LugerNedic2019,Vladimirov1969,Vladimirov1979,VladimirovDrozzinov1974}, representations using distributions, \cf \cite{Zemanian1965}, or representations using positive semi-definite functions, \cf \cite{Nedic2021b,Pick1915}.

For functions of one variable, the study of operator representations of Herglotz-Nevanlinna functions dates back around 100 years, though the results were not formulated in the language of operator theory. An early result of Nevanlinna, written using modern notation, shows that given a function $q\colon \C^+ \to \C$ there exists a finite positive Borel measure $\mu$ on $\R$ such that $q(z) = \int_\R(t-z)^{-1}\diff\mu(t)$ if and only if $q$ is holomorphic, $\Im[q(z)] \geq 0$ for all $z \in \C^+$ and $\liminf_{y \to \infty}y|q(\I\,y)| < \infty$ \cite{AglerEtal2016,Nevanlinna1922}. In the language of operator theory, this result can be reformulated in the following way. Given a function $q\colon \C^+ \to \C$ there exists a Hilbert space $(\hilb,\langle\,\cdot\,,\,\cdot\,\rangle_\hilb)$ a bounded self-adjoint operator in $\hilb$ such that $q(z) = \langle(A-z)^{-1}v,v\rangle_\hilb$ if and only if the same three conditions as before are satisfied. This representation was later extended to cover the whole class of Herglotz-Nevanlinna function in one variable as well as various generalizations, see \eg
\cite{KreinLanger1977,LangerTextorius1977}.

The results of this paper concern operator representations of Herglotz-Nevanlinna functions of several variables. More precisely, the main theorem states that a function $q\colon\C^{+n} \to \C$ is a Herglotz-Nevanlinna function if and only if it holds, for all $\vec{z},\vec{w} \in \C^{+n}$, that
\begin{multline*}
q(\vec{z}) = \bar{q(\vec{w})} \\ + \frac{1}{(2\,\I)^{n-1}}\prod_{k=1}^n(z_k-\bar{w_k}) \cdot \bigg\langle\prod_{k=1}^n(\Id + (z_k - w_k)(A_k - z_k)^{-1})\,v_{\vec{w}},v_{\vec{w}}\bigg\rangle_\hilb,
\end{multline*}
where $(\hilb,\langle\cdot,\cdot\rangle_\hilb)$ is a complex Hilbert space, $A_1,\ldots,A_n$ are self-adjoint linear relations in $\hilb$ and $v_{\vec{w}} \in \hilb$, \cf Theorem \ref{thm:opRep_Nvar}. A refinement of this result is presented in Theorem \ref{thm:opRep_alt}.

The structure of this paper is as follows. After the introduction, we focus, in Section \ref{sec:linear_relations}, on reviewing some basic facts about linear relations. Section \ref{sec:about_HN_functions} reviews instead the basic existing representations of Herglotz-Nevanlinna functions that are used later on, and Section \ref{sec:opRep} focuses on the main result of this paper, its proof and interpretations. Finally, Section \ref{sec:comparisons} reviews some previous results in the area, particularly a representation for the class of Loewner functions. Additionally, Subsection \ref{sec:rep_kernel} provides a description of the reproducing kernel Hilbert space corresponding to a Herglotz-Nevanlinna function of several variables, as follows from the proof of the main result.

\section{Linear relations on a Hilbert space}
\label{sec:linear_relations}

In this section, we review some basic facts and definitions that are used throughout the later sections.
For a comprehensive introduction to the topic of linear relations in a general setting we refer to \eg \cite[Ch. 1]{BehrndtHassiDeSnoo2020}, which is used as the main point reference in this section. A few other references on the subject include \eg
\cite{Arens1961,DijksmaDeSnoo1974,DijksmaDeSnoo1987,LangerTextorius1977}. 

For the purposes of this paper, it is sufficient to consider subspaces of the product of a complex Hilbert space with itself. Hence, consider the following definition, \cf \cite[pg. 12]{BehrndtHassiDeSnoo2020}.

\begin{define}
\label{def:linear_relation}
Let $(\hilb,\langle\cdot,\cdot\rangle_\hilb)$ be a complex Hilbert space. A \emph{linear relation} $A$ in $\hilb$ is a linear subspace $A \subseteq \hilb \times \hilb$, \ie for every $\{x_1,y_1\},\{x_2,y_2\} \in A$ and every $\alpha,\beta \in \C$ it holds that $\{\alpha x_1+ \beta x_2,\alpha y_1+ \beta y_2\} \in A$.
\end{define}

For any linear relation $A$ in $\hilb$, the following four sets are of fundamental importance:
$$\begin{array}{RRCL}
\text{domain:} & \Dom(A) & := & \{x\in \hilb~|~\exists\,y \in \hilb: \{x,y\} \in A\}, \\
\text{range:} & \Ran(A) & := & \{y\in \hilb~|~\exists\,x \in \hilb: \{x,y\} \in A\}, \\ 
\text{kernel:} & \ker(A) & := & \{x\in \hilb~|~\{x,0\} \in A\}, \\
\text{multi-valued part:} & \Mult(A) & := & \{y\in \hilb~|~\{0,y\} \in A\}.
\end{array}$$
The closure of the domain of $A$ is denoted by $\bar{\Dom(A)}$ and refers to the closure of $\Dom(A) \subseteq \hilb$ in the topology induced by the inner product on $\hilb$. The closures of the other three sets are defined analogously.

If $B$ is a linear operator on $\hilb$ in the usual sense, it is, in this setting, identified with its graph $\Gamma(B) := \{\{x,y\} \in \hilb \times \hilb~|~Bx=y\}$. More generally, a linear relation $A$ equals the graph of some linear operator, and may hence be identified with said operator, if and only if $\Mult(A) = \{0\}$ \cite[pg. 12]{BehrndtHassiDeSnoo2020}.

A linear relation $A$ is called \emph{bounded} is there exists a constant $C \geq 0$ such that $\|y\|_\hilb \leq C\,\|x\|_\hilb$ for every pair $\{x,y\} \in A$. Note that if $\{0,y\} \in A$ and $A$ is bounded, then $\|y\|_\hilb \leq 0$, implying that $\Mult(A) = \{0\}$ in this case. Furthermore, the relation $A$ is called \emph{closed} if it is closed as a subspace of $\hilb \times \hilb$.

For any linear relation $A$, its \emph{inverse} is the linear relation $A^{-1}$ defined by the equivalence $\{x,y\} \in A^{-1} \iff \{y,x\} \in A$. Furthermore, given a number $\theta \in \C$, the \emph{resolvent relation} is the relation $(A-\theta)^{-1}$. Here, as often, it is implied that $(A - \theta)^{-1} := (A - \theta\,\Id)^{-1}$ where $\Id$ denotes the identity relation on $\hilb$. In connection to resolvent relations, we also define the \emph{resolvent set} $\rho(A)$ of a linear relation $A$ as
$$\rho(A) := \{\theta \in \C~|~(A-\theta)^{-1}\:\text{is bounded and}\:\bar{\Ran(A-\theta)} = \hilb\},$$
\cf \cite[Def. 1.2.3]{BehrndtHassiDeSnoo2020}. Different resolvent relations of $A$ are connected via the following identity, \cf \cite[Thm. 1.2.6 and Cor. 1.2.7]{BehrndtHassiDeSnoo2020}.

\begin{lemma}[Resolvent identity]
\label{lem:resolvent_identity}
Let $A$ be a closed linear relation on $\hilb$. Then, $\rho(A)$ is an open subset of $\C$ and for every $\theta,\tau \in \rho(A)$ it holds that
\begin{equation}
    \label{eq:resolvent_identity}
    (A-\theta)^{-1} - (A-\tau)^{-1} = (\theta-\tau)(A-\theta)^{-1}(A-\tau)^{-1}.
\end{equation}
Furthermore, the function
$$\theta \mapsto (A-\theta)^{-1}$$
is a holomorphic as a Banach space valued function, \cf \cite[Chap. 4]{Dineen1981}.
\end{lemma}

We note that a variant of identity \eqref{eq:resolvent_identity} can also be stated for non-closed relations as well as points not necessarily in the resolvent set, \cf \cite[Prop. 1.1.7 and Thm 1.2.6]{BehrndtHassiDeSnoo2020}. Finally, let us recall the definition of the adjoint relation, \cf \cite[Def. 1.3.1]{BehrndtHassiDeSnoo2020}.

\begin{define}
\label{def:adjoint_relation}
Let $A$ be a linear relation in $\hilb$. Then, its \emph{adjoint relation} $A^*$ is the linear relation defined via the equivalence
$$\{v,w\} \in A^* \iff \forall\,\{x,y\}\in A:\:\langle y,v \rangle_\hilb = \langle x,w \rangle_\hilb.$$
Moreover, the relation $A$ is called \emph{self-adjoint} if $A = A^*$ as subsets of $\hilb \times \hilb$.
\end{define}

The following properties of self-adjoint relations are going to be used later on, \cf \cite[Sec. 1.3 and 1.5]{BehrndtHassiDeSnoo2020}.

\begin{lemma}
\label{lem:self_adjoint}
Let $A$ be a self-adjoint relation on $\hilb$. Then, $A$ is necessarily closed, it holds that $\C\setminus\R \subseteq \rho(A)$ and for each $z \in \rho(A)$ the resolvent relation $(A-z)^{-1}$ is a bounded everywhere-defined linear operator.
\end{lemma}

\section{The one variable case}
\label{sec:one_variable_case}

As mentioned in the introduction, a classical result dating back to the 1920s establishes an equivalence between the analytic properties of a function $q\colon\C \to \C^+$, its integral representation in terms a certain Borel measure, and its operator representation in terms of a certain self-adjoint operator. Using the language and notation that is used throughout the rest of the paper, we summarize, for a general Herglotz-Nevanlinna function, the corresponding results in the following theorem, \cf \cite{Cauer1932,KacKrein1974,KreinLanger1977,LangerTextorius1977,Nevanlinna1922}.

\begin{thm}\label{thm:all_rep_1var}
The following statements concerning a function $q \colon \C^+ \to \C$ are equivalent.
\begin{itemize}
    \item[(a)]{The function $q$ is a Herglotz-Nevanlinna function, \ie it is holomorphic with non-negative imaginary part.}
    \item[(b)]{There exists a number $a \in \R$, a number $b \geq 0$, and a positive Borel measure $\mu$ on $\R$ with $\int_\R(1+t^2)^{-1}\diff\mu(t) < \infty$ such that it holds, for all $z \in \C^+$, that
    \begin{equation}\label{eq:intRep_1var}
        q(z) = a + b\,z + \frac{1}{\pi}\int_\R\left(\frac{1}{t-z}-\frac{t}{1+t^2}\right)\diff\mu(t).
    \end{equation}
    }
    \item[(c)]{There exists a Hilbert space $(\hilb,\langle\cdot,\cdot\rangle_\hilb)$ and a self-adjoint linear
  relation $A$ in $\hilb$ such that to every $w \in \C^+$ corresponds
  an element $v_{w} \in \hilb$ such that it holds, for all
  $z \in \C^+$, that
\begin{equation}
    \label{eq:opRep_1var}
    q(z) = \bar{q(w)} + (z - \bar{w})\big\langle(\Id + (z - w)(A - z)^{-1})v_w,v_w\big\rangle_\hilb.
  \end{equation}
  }
\end{itemize}
\end{thm}

Regarding the integral representation \eqref{eq:intRep_1var}, various generalization for functions of several variables could first be found in \eg \cite{KoranyiPukanszky1963,Vladimirov1969,Vladimirov1979}. The most direct possible analogue to representation \eqref{eq:intRep_1var} in terms of its form and the objects used in said representation can be found in \cite[Thm. 3.1]{LugerNedic2017} for the case of functions of two variables, and \cite[Thm. 4.1]{LugerNedic2019} which covers the general several variable case.

For functions of one variable, various operator representation results are often written for larger classes of functions, see \eg \cite[Satz 1.4]{KreinLanger1977} and \cite[Thm. 2.2]{LangerTextorius1977}, with representation \eqref{eq:opRep_1var} being a special case of these for ordinary Herglotz-Nevanlinna functions in particular. For function of several variables, a generalization can be found in \cite[Thm. 1.9]{AglerEtal2016} (\cf Theorem \ref{thm:loewner_rep}). However, the most direct possible analogue to representation \eqref{eq:opRep_1var} in terms of its form and the objects used in said representation is Theorem \ref{thm:opRep_Nvar} later on (see also Section \ref{sec:loewner_rep} for further discussions and comparisons).

Given a Herglotz-Nevanlinna function $q$ in one variable, it turns out that the Hilbert space $\hilb$, the linear relation $A$, the number $w$ and the vector $v \in \hilb$ can be specified explicitly, with one option being with the help of the representing parameters $(a,b,\mu)$ from representation \eqref{eq:intRep_1var}. However, to represent a functions whose integral representation \eqref{eq:intRep_1var} has a non-zero linear term $b$, self-adjoint relations must be used in place of self-adjoint operators. To compare, we note that when $b=0$, the Hilbert space in representation \eqref{eq:opRep_1var} can be chosen as $\hilb = \Leb^2(\R,\mu)$, equipped with
$$\langle f_1,f_2 \rangle_\hilb := \frac{1}{\pi}\int_\R f_1(t)\bar{f_2(t)}\diff t.$$
The linear relation $A$ (which is a regular linear operator in this case) can be chosen as $(Af)(t) := t\:f(t)$ on the domain
$$\Dom(A) := \{f \in \Leb^2(\R,\mu)~|~t \mapsto t\:f(t) \in \Leb^2(\R,\mu)\}.$$
Finally, given $w \in \C$ the element $v_w \in \hilb$ can be chosen as $v_w \colon t \mapsto (t-w)^{-1}$. With these choices, one can verify that writing out representation \eqref{eq:opRep_1var} reproduces representation \eqref{eq:intRep_1var} in this case.

On the other hand, if $b > 0$, the Hilbert space should instead be chosen as $\hilb = \Leb^2(\R,\mu) \oplus \C$, equipped with the inner product
$$\left\langle\relone{f_1}{\xi_1},\relone{f_2}{\xi_2}\right\rangle_\hilb := \frac{1}{\pi}\int_\R f_1(t)\bar{f_2(t)}\diff t + b\,\xi_1\,\bar{\xi_2}.$$
The linear relation $A$ is then chosen as
$$\relpair{f}{\xi}{g}{\eta} \in A \iff g(t) = t\:f(t) \quad\wedge\quad \xi = 0,$$
and can be verified to be self-adjoint directly by definition. Finally, given $w \in \C$ the element $v_w \in \hilb$ can be chosen as 
$$v := \relone{t \mapsto (t-w)^{-1}}{1} \in \hilb.$$
As in the previous case, a direct computation made by writing out representation \eqref{eq:opRep_1var} in terms of these choices reproduces representation \eqref{eq:intRep_1var}. Finally, note that the value of the parameter $b$ can be obtained purely from the properties of the function $q$, namely
$$b = \lim\limits_{z \ntto \infty}\frac{q(z)}{z},$$
where $z \ntto \infty$ denotes the limit $|z| \to \infty$ taken in any upper Stoltz domain \cite{KacKrein1974}.

The type of construction as described above would constitute the proof of the implication (b) $\Rightarrow$ (c) of Theorem \ref{thm:all_rep_1var}. The proof that (c) $\Rightarrow$ (a) amounts to checking that the imaginary part of a function defined by representation \eqref{eq:opRep_1var} is non-negative, while the holomorphicity of such a function follows from Lemma \ref{lem:resolvent_identity}. Finally, the implication (a) $\Rightarrow$ (b) is (half of) the classical integral representation theorem for Herglotz-Nevanlinna functions, \cf \cite{KacKrein1974}.

\section{Herglotz-Nevanlinna functions in several variables}
\label{sec:about_HN_functions}

In this section, we review the facts and theorems about Herglotz-Nevanlinna functions in several variables that are prerequisite for the upcoming results. First, the focus lies on the integral representation theorem, and later on the representation using positive semi-definite functions.

\subsection{Integral representation theorem}
\label{subsec:int_rep}

A fundamental tool in the study of Herglotz-Nevanlinna functions is their integral representation theorem, the statement of which requires us to first introduce some notation. Given ambient numbers $z \in \cut$ and $t \in \R$, we define first the expressions
\begin{multline*}
    N_{-1}(z,t) := \frac{1}{2\,\I}\left(\frac{1}{t - z} - \frac{1}{t - \I}\right), \quad N_{0}(z,t) := \frac{1}{2\,\I}\left(\frac{1}{t - \I} - \frac{1}{t + \I}\right), \\
    N_{1}(z,t) := \frac{1}{2\,\I}\left(\frac{1}{t + \I} - \frac{1}{t - \bar{z}}\right).
\end{multline*}
Note that $N_0(z,t) \in \R$ and is independent of $z \in \cut$, while
$$\bar{N_{-1}(z,t)} = N_{1}(z,t)$$
for all $z \in \cut$ and $t \in \R$. Define now the kernel $K_n \colon \cutN \times \R^n \to \C^n$ as
\begin{multline}\label{eq:kernel_Kn}
    K_n(\vec{z},\vec{t}) := \I\bigg(2\prod_{j=1}^n\big(N_{-1}(z_j,t_j)+N_{0}(\I,t_j)\big)-\prod_{j=1}^nN_{0}(\I,t_j)\bigg) \\
    = \I\left(\frac{2}{(2\,\I)^n}\prod_{j=1}^n\left(\frac{1}{t_j-z_j}-\frac{1}{t_j+\I}\right)-\frac{1}{(2\,\I)^n}\prod_{j=1}^n\left(\frac{1}{t_j-\I}-\frac{1}{t_j+\I}\right)\right).
\end{multline}

We may now recall the integral representation theorem for Herglotz-Nevanlinna functions of several variables, \cf \cite[Thm. 4.1]{LugerNedic2019}.

\begin{thm}\label{thm:intRep_Nvar}
A function $q\colon \C^{+n} \to \C$ is a Herglotz-Nevanlinna function if and only if $q$ can be written as
\begin{equation}\label{eq:intRep_Nvar}
q(\vec{z}) = a + \sum_{\ell=1}^nb_\ell z_\ell + \frac{1}{\pi^n}\int_{\R^n}K_n(\vec{z},\vec{t})\diff\mu(\vec{t}),
\end{equation}
where $a \in \R$, $\vec{b} \in [0,\infty)^n$ and $\mu$ is a positive Borel measure on $\R^n$ satisfying the growth condition
\begin{equation}
\label{eq:measure_growth}
\int_{\R^n}\prod_{\ell=1}^n\frac{1}{1+t_\ell^2}\diff\mu(\vec{t}) < \infty
\end{equation}
and the Nevanlinna condition
\begin{equation}
\label{eq:measure_Nevan}
\sum_{\substack{\vec{\rho} \in \{-1,0,1\}^n \\ -1\in\vec{\rho} \wedge 1\in\vec{\rho}}}\int_{\R^n}\prod_{j=1}^nN_{\rho_j}(z_j,t_j)\diff\mu(\vec{t}) = 0
\end{equation}
for all $\vec{z} \in \C^{+n}$. Furthermore, for a given function $q$, the triple of representing parameters $(a,\vec{b},\mu)$ is unique.
\end{thm}

We note that Borel measures on $\R^n$ satisfying conditions \eqref{eq:measure_growth} and \eqref{eq:measure_Nevan} are called \emph{Nevanlinna measures}, \cf \cite{LugerNedic2021,NedicSaksman2024}.
 
\subsection{Representation using positive semi-definite functions}

We recall first the following definition, \cf \cite[pg. 3002]{AglerEtal2016}.

\begin{define}\label{def:positive_semi_def_function}
Let $\Omega \subseteq \C^n$. A function $F\colon \Omega \times \Omega \to \C$ is called \emph{positive semi-definite} if for all $m \geq 1$ and every choice of $m$ vectors $\vec{z}_1,\vec{z}_2,\ldots,\vec{z}_m \in \Omega$ and $m$ numbers $c_1,c_2,\ldots,c_m \in \C$ it holds that
\begin{equation}
    \label{eq:positive_semi_def_function}
    \sum_{i,j = 1}^mF(\vec{z}_i,\vec{z}_j)\,c_i\,\bar{c_j} \geq 0.
\end{equation}
\end{define}

A particular class of positive semi-definite functions that plays an important role later on is the following, \cf \cite[Def. 3.5]{Nedic2021b}.

\begin{define}\label{def:poisson_type}
A function $F\colon \C^{+n} \times \C^{+n} \to \C$ is called a \emph{Poisson-type function} if there exists a positive Borel measure $\mu$ on $\R^n$ satisfying the growth condition \eqref{eq:measure_growth} such that
$$F(\vec{z},\vec{w}) = \frac{1}{\pi^n}\int_{\R^n}\prod_{\ell = 1}^n\frac{1}{(t_\ell - z_\ell)(t_\ell - \bar{w_\ell})}\diff\mu(\vec{t})$$
for all $\vec{z},\vec{w} \in \C^{+n}$. In this case, we say that the function $F$ is given by the measure $\mu$.
\end{define}

Poisson-type functions are readily shown to be positive semi-definite \cite[Lem. 3.6]{Nedic2021b} and the measure corresponding to a particular function is uniquely determined \cite[Lem 3.8]{Nedic2021b}. The following theorem now gives another characterization of Herglotz-Nevanlinna functions, \cf \cite[Thm. 4.1]{Nedic2021b}.

\begin{thm}\label{thm:positive_semi_decomposition}
Let $n \in \N$ and let $q\colon \C^{+n} \to \C$ be a holomorphic function. Then, $q$ is a Herglotz-Nevanlinna function if and only if there exists a vector $\vec{d} \in [0,\infty)^n$ and a Poisson-type function $D$ (given by some measure $\nu$) such that the equality
\begin{equation}
    \label{eq:positive_semi_decomposition}
    q(\vec{z}) - \bar{q(\vec{w})} = \sum_{j=1}^nd_j(z_j - \bar{w_j}) + \frac{1}{(2\,\I)^{n-1}}\prod_{j=1}^n(z_j - \bar{w_j})\cdot D(\vec{z},\vec{w})
\end{equation}
holds for all $\vec{z},\vec{w} \in \C^{+n}$.

Furthermore, let $(a,\vec{b},\mu)$ be the representing parameters of the function $q$ in the sense of Theorem \ref{thm:intRep_Nvar}. Then, one may choose $\vec{d} = \vec{b}$ and $D$ as the Poisson-type function given by the representing measure $\mu$.
\end{thm}

We note that \cite[Thm. 4.1]{Nedic2021b} is formulated more generally than Theorem \ref{thm:positive_semi_decomposition}. However, this weaker statement presented here is sufficient for the purposes of this paper.

Finally, let us review an elementary but important example of a Poisson-type function.

\begin{example}
\label{ex:Poisson_type_Lebsgue}
Let $\lambda_{\R^n}$ denote the standard Lebesgue measure of $\R^n$. To determine the Poisson-type function $F\colon \C^{+n} \times \C^{+n} \to \C$ given by the Lebesgue measure we note first that by standard residue calculus, it holds for any $z,w \in \C^+$ that
$$\frac{1}{\pi}\int_\R\frac{1}{(t-z)(t-\bar{w})}\diff t = \frac{2\,\I}{z-\bar{w}}.$$
A repeated application of this calculation yields
$$F(\vec{z},\vec{w}) = \frac{1}{\pi^n}\int_{\R^n}\prod_{j=1}^n\frac{1}{(t_j-z_j)(t_j-\bar{w_j})}\diff\lambda_{\R^n}(\vec{t}) = \prod_{j=1}^n\frac{2\,\I}{z_j - \bar{w_j}}.$$
It is also important to note that, for example, the function $\til{F}\colon\C^{+n} \times \C^{+n} \to \C$ defined as
$$\til{F}(\vec{z},\vec{w}) := \prod_{j=1}^{n-1}\frac{2\,\I}{z_j - \bar{w_j}}$$
is not a Poisson-type function. If it were, one can show as in \cite[Ex. 4.12]{Nedic2021b} that, in that case, it would be given by the zero-measure, an impossibility. Nevertheless, the function $\til{F}$ is still positive semi-definite on $\C^{+n} \times \C^{+n}$. \hfill $\lozenge$ 
\end{example}

\section{Operator representation theorems}
\label{sec:opRep}

Before we state our main result, let us outline the main idea of how the construction of a direct analogue to representation \eqref{eq:opRep_1var} can be carried out in several variables. For the moment, suppose we restrict our attention to functions $q$ which have $\vec{b} = \vec{0}$, so that representation \eqref{eq:intRep_Nvar} takes the form
\begin{equation}
\label{e.integral}
    q(\vec{z}) = a + \frac{1}{\pi^n}\int_{\R^n} K_n(\vec{z}, \vec{t})\diff\mu(\vec{t}),
\end{equation}
where $a \in \R$ and $\mu$ satisfies conditions \eqref{eq:measure_growth} and \eqref{eq:measure_Nevan}. To express
\eqref{e.integral} in terms of linear relations (or operators in this
case), we note that
\begin{equation*}
  K_n(\vec{z}, \vec{t})-\overline{K_n(\I\,\vec{1}, \vec{t})} =\frac{2\,\I}{(2\,\I)^n}\prod_{k=1}^n\left(\frac{1}{t_k-z_k}-\frac{1}{t_k+\I}\right) = \frac{2\,\I}{(2\,\I)^n}\prod_{k=1}^n\frac{(z_k+\I)(t_k-\I)}{(t_k-z_k)(t_k^2+1)}.
\end{equation*}
Here, $\I\,\vec{1}$ is a shorthand for $(\I,\ldots,\I)$.  Now let $\hilb=\Leb^2(\R^n,\mu)$ with $\langle{}f,g\rangle{}_\hilb=\frac{1}{\pi^n}\int f\bar g\,\diff\mu$, and define $(A_kf)(\vec{t})=t_kf(\vec{t})$ and $v(\vec{t})=\prod_{k=1}^n(t_k-\I)^{-1}$. The representation \eqref{e.integral} then implies that
\begin{equation}
\label{e.myrep}
    q(\vec{z}) = \overline{q(\I\,\vec{1})}+\left\langle\frac{2\,\I}{(2\,\I)^n}\prod_{k=1}^n(z_k+\I)(A_k-\I)(A_k-z_k)^{-1}v,v\right\rangle_\hilb.
\end{equation}
Because $\Id+(z-w)(A-z)^{-1}=(A-w)(A-z)^{-1}$ when $A$ is an operator, \eqref{e.myrep} gives an analogue of \eqref{eq:opRep_1var} in several variables, at least when $\vec{w}=\I\,\vec{1}$. To obtain a result for arbitrary $\vec{w}$, the representation \eqref{eq:intRep_Nvar} must be replaced with that of Theorem \ref{thm:positive_semi_decomposition}. This leads to our main result, which gives the following characterization of Herglotz-Nevanlinna functions in several variables.

\begin{thm}
\label{thm:opRep_Nvar}
A function $q\colon\C^{+n} \to \C$ is a Herglotz-Nevanlinna function if and only if there exists a complex Hilbert space $(\hilb,\langle\cdot,\cdot\rangle_\hilb)$ and $n$ self-adjoint linear relations $A_1,\ldots,A_n$ in $\hilb$, such that to every   $\vec w \in \C^{+n}$ corresponds an element $v_{\vec w}\in\hilb$ for which
\begin{multline}\label{eq:opRep_Nvar}
q(\vec{z}) = \bar{q(\vec{w})} \\ + \frac{1}{(2\,\I)^{n-1}}\prod_{k=1}^n(z_k-\bar{w_k}) \cdot \bigg\langle\prod_{k=1}^n(\Id + (z_k - w_k)(A_k - z_k)^{-1})\,v_{\vec{w}},v_{\vec{w}}\bigg\rangle_\hilb
\end{multline}
for all $\vec z\in\C^{+n}$, where $\Id$ denotes the identity operator on $\hilb$.
\end{thm}

Before proceeding with the proof, we first list a few remarks regarding the formulation of the theorem and the objects that appear in it.

\begin{remarks}
  (\emph{a}) Given a linear relation $A$ on $\hilb$ and
  $z,w \in \C^+$, the expression $\Id + (z-w)(A-z)^{-1}$ is a special
  case of a \emph{Möbius transform} of $A$
  \cite[pg. 20]{BehrndtHassiDeSnoo2020}.

  (\emph{b}) On a notational level, it is understood that
  \begin{multline*}
    \prod_{k=1}^n(\Id + (z_k - w_k)(A_k - z_k)^{-1})\,v_{\vec{w}} \\
    = (\Id + (z_1 - w_1)(A_1 - z_1)^{-1})\ldots(\Id + (z_n - w_n)(A_n - z_n)^{-1})v_{\vec{w}}.
  \end{multline*}

  (\emph{c}) Theorem \ref{thm:opRep_Nvar} makes no claims on the
  uniqueness of representation \eqref{eq:opRep_Nvar}, nor does it
  include any kind of characterization of possible spaces $\hilb$,
  relations $A_k$ etc. The proof, however, is constructive and gives
  an example of a possible Hilbert space and corresponding
  self-adjoint relations.
\end{remarks}

The proof of Theorem \ref{thm:opRep_Nvar} is split among Proposition \ref{prop:hn_to_op} and \ref{prop:op_to_hn} as they are both interesting in their own right, and provide an opportunity to contrast Theorem \ref{thm:opRep_Nvar} against both Theorem \ref{thm:all_rep_1var}.c and Theorem \ref{thm:intRep_Nvar}. 

\subsection{Constructing the linear relations from a given function}

First, we show how we can give an explicit construction of a Hilbert space and accompanying self-adjoint relations given a Herglotz-Nevanlinna function.

\begin{prop}\label{prop:hn_to_op}
Let $q\colon\C^{+n} \to \C$ be a Herglotz-Nevanlinna function. Then, there exists a complex Hilbert space $(\hilb,\langle\cdot,\cdot\rangle_\hilb)$ and $n$ self-adjoint linear relations $A_1,\ldots,A_n$ in $\hilb$ such that \eqref{eq:opRep_Nvar} holds for some $\vec{w} \mapsto v_{\vec{w}}$. Moreover, one may arrange that
\begin{equation}\label{eq:admissible-v_w}
    v_{\vec{z}} = \prod_{k=1}^n(\Id + (z_k - w_k)(A_k - z_k)^{-1})\,v_{\vec{w}}
\end{equation}
for all $\vec{z}, \vec{w} \in \C^{+n}$.
\end{prop}

\proof
Let $q$ be a Herglotz-Nevanlinna function. By Theorem \ref{thm:intRep_Nvar}, it has a representation, for $\vec{z} \in \C^{+n}$, of the form
$$q(\vec{z}) = a + \sum_{\ell=1}^nb_\ell z_\ell + \frac{1}{\pi^n}\int_{\R^n}K_n(\vec{z},\vec{t})\diff\mu(\vec{t}),$$
where $K_n$ is specified in formula \eqref{eq:kernel_Kn}, $a \in \R$, $\vec{b} \in [0,\infty)^n$ and $\mu$ is a positive Borel measure on $\R^n$ satisfying conditions \eqref{eq:measure_growth} and \eqref{eq:measure_Nevan}. The proof is now be conducted by explicitly specifying $\hilb, A_k$ and $v_{\vec{w}}$, and showing that our choice is valid. To that end, we are going to fragment the proof into smaller cases depending on which of the parameters $\vec{b}$ and $\mu$ are zero. The structure of these cases is summarized in Table \ref{tab:cases} below. Simultaneously, note that the parameter $a$ has no effect, regardless of whether it is zero or not, on the right-hand side of formula \eqref{eq:opRep_Nvar}.

\begin{table}[!ht]
    \centering
    \begin{tabular}{c|cc}
Case & $\vec{b} = \vec{0}$? & $\mu = 0$? \\
\hline
1 & \checkmark & \checkmark \\ 
2 & \checkmark & $\times$ \\ 
3 & $\times$ & \checkmark \\ 
4 & $\times$ & $\times$
    \end{tabular}
    \caption{A summary of the different cases of Part 1 of the proof of Theorem \ref{thm:opRep_Nvar}.}
    \label{tab:cases}
\end{table}

\textit{Case 1:} $\vec{b} = \vec{0}, \mu = 0$. In this case, it holds that $q(\vec{z}) - \bar{q(\vec{w})} = 0$ for all $\vec{z},\vec{w} \in \C^{+n}$. Hence, the right-hand side of formula \eqref{eq:opRep_Nvar} must also produce the value zero. This is easiest achieved by choosing $(\hilb,\langle\cdot,\cdot\rangle_\hilb)$ as the trivial vector space over $\C$, \ie $\hilb := \{0\}$ with $\langle 0,0 \rangle_\hilb := 0$. The relations $A_k$ may then be chosen as $A_k := \Id$ for $k \in \{1,\ldots,n\}$ and the element $v_{\vec{w}}$ as $v_{\vec{w}} := 0$. This concludes Case 1.

\textit{Case 2:} $\vec{b} = \vec{0}, \mu \neq 0$. In this case, we choose the space $\hilb := \Leb^2(\R^n,\mu)$ as a vector space over $\C$, where $\mu$ is the representing measure of the function $q$. Define now an inner product on $\hilb$ by setting, for $f,g \in \hilb$,
\begin{equation}
    \label{eq:case12_inner_product}
    \langle f,g \rangle_\hilb := \frac{1}{\pi^n}\int_{\R^n} f(\vec{t})\bar{g(\vec{t})}\diff\mu(\vec{t}).
\end{equation}
Clearly, $(\hilb,\langle\cdot,\cdot\rangle_\hilb)$ is a Hilbert
space. For each $k = 1,\ldots,n$ define now the linear relation $A_k$
(in this case an operator) via the equivalence
$$\{f,g\} \in A_k \iff t_k\,f(\vec{t}) = g(\vec{t}) \quad\text{a.e. on}\;\R^n.$$

An elementary computation now yields that each relation $A_k$ is self-adjoint. Moreover, by Lemma \ref{lem:self_adjoint}, it holds for each $z \in \C^+$ that the resolvent relation $(A_k-z)^{-1}$ is an everywhere-defined bounded linear operator. Explicitly,
$$\big((A_k-z)^{-1}f\big)(\vec{t}) = (t_k-z)^{-1}f(\vec{t}).$$

Let now $\vec{z},\vec{w} \in \C^{+n}$ be arbitrary and let them be fixed for the remainder of Case 2. Define a function $v_{\vec{w}}\colon\R^n \to \C$ by setting $v_{\vec{w}}(\vec{t}) := \prod_{j=1}^n(t_j-w_j)^{-1}$. Note that $v_{\vec{w}} \in \hilb$ since
\begin{multline}
\label{eq:checking_L2}
    \int_{\R^n}|v_{\vec{w}}(\vec{t})|^2\diff\mu(\vec{t}) \leq \bigg\|\vec{t} \mapsto \prod_{j=1}^n\frac{1+t_j^2}{|t_j-w_j|^2}\bigg\|_\infty\cdot\int_{\R^n}\prod_{j=1}^n\frac{1}{1+t_j^2}\diff\mu(\vec{t}) < \infty
\end{multline}
as the measure $\mu$ satisfies the growth condition \eqref{eq:measure_growth}. Here, $\|\cdot\|_\infty$ denotes the standard supremum norm for functions defined on $\R^n$. Next, after all of the above choices of $\hilb,A_k,v_{\vec{w}}$ have been made, the expression
\begin{equation}
    \label{eq:mobious_expression}
    \prod_{k=1}^n(\Id + (z_k - w_k)(A_k - z_k)^{-1})\,v_{\vec{w}}
\end{equation}
is well-defined since all of the different resolvent relations $(A_k - z_k)^{-1}$ are everywhere-defined bounded linear operators.

Continue by considering the expression
$$\frac{1}{(2\,\I)^{n-1}}\prod_{k=1}^n(z_k-\bar{w_k})\cdot \bigg\langle\prod_{k=1}^n(\Id + (z_k - w_k)(A_k - z_k)^{-1})\,v_{\vec{w}},v_{\vec{w}}\bigg\rangle_\hilb = (*1).$$
To that end, note first that
\begin{equation}
    \label{eq:action_1}
    (\Id + (z_k - w_k)(A_k - z_k)^{-1})v_{\vec{w}}(\vec{t}) = \left(1 + \frac{z_k - w_k}{t_k - z_k}\right)v_{\vec{w}}(\vec{t}).
\end{equation}
Making repeated use of equality \eqref{eq:action_1}, we calculate further that
$$\begin{array}{RCL}
    (*1) & = & \frac{1}{(2\,\I)^{n-1}}\prod_{k=1}^n(z_k - \bar{w_k}) \\[0.45cm]
    ~ & ~ & \cdot\,\frac{1}{\pi^n}\int_{\R^n}\prod_{k=1}^n\left(1+\frac{z_k - w_k}{t_k-z_k}\right)\prod_{k=1}^n\frac{1}{t_k-w_k}\prod_{k=1}^n\frac{1}{t_k-\bar{w_k}}\diff\mu(\vec{t}) \\[0.45cm]
    ~ & = & \frac{1}{(2\,\I)^{n-1}}\prod_{k=1}^n(z_k - \bar{w_k})\cdot\frac{1}{\pi^n}\int_{\R^n}\prod_{k=1}^n\frac{1}{(t_k-z_k)(t_k-\bar{w_k})}\diff\mu(\vec{t}) \\[0.45cm]
    ~ & = & \frac{1}{(2\,\I)^{n-1}}\prod_{k=1}^n(z_k - \bar{w_k})\,D(\vec{z},\vec{w}) = (**1),
\end{array}$$
where $D$ is the Poisson-type function given by the measure $\mu$. As we are in the case where $\vec{b} = \vec{0}$, Theorem \ref{thm:positive_semi_decomposition} now implies that
$$(**1) = q(\vec{z}) - \bar{q(\vec{w})},$$
thereby leading to equality \eqref{eq:opRep_Nvar}.  Moreover, we see
that \eqref{eq:admissible-v_w} holds.  This concludes Case 2.

\textit{Case 3:} $\vec{b} \neq \vec{0}, \mu = 0$. 
To begin, given $j \in \{1,\ldots,n\}$, define the space $R_j$ by setting $R_j := \vec{e}_j^\perp \cong \R^{n-1}$ as the orthogonal complement to the standard basis vector $\vec{e}_j$. Equip this space with the coordinates $\vec{s} := (s_1,\ldots,s_{j-1},s_{j+1},\ldots,s_n)$ and the measure
$$\diff\lambda_{R_j}(\vec{s}) := b_j\,\diff s_1 \,\ldots\, \diff s_{j-1}\,\diff s_{j+1} \,\ldots\, \diff s_n,$$
choosing the space $\hilb$ as
$$\hilb := \bigoplus_{j = 1}^n\Leb^2(R_j,\lambda_{R_j}).$$
Note that if $j \in \{1,\ldots,n\}$ is such that $b_j = 0$, then $\Leb^2(R_j,\lambda_{R_j}) \cong \{0\}$, \ie all functions lie in the equivalence class of the constant-zero function, while if $b_j \neq 0$, then $\Leb^2(R_j,\lambda_{R_j}) \cong \Leb^2(\R^{n-1},\lambda_{\R^{n-1}})$, where, as before, $\lambda_{\R^{n-1}}$ denotes the standard Lebesgue measure on $\R^{n-1}$. 

Set now the notation that an element $\vec{\xi} \in \hilb$ is written as a vector of functions, \ie
$$\vec{\xi} = (\xi_1,\xi_2,\ldots,\xi_n) \quad\text{where}\quad \xi_j \in \Leb^2(R_j,\lambda_{R_J}).$$
We define an inner product on $\hilb$ by setting, for $\vec{\xi},\vec{\eta} \in \hilb$, as
$$\langle\vec{\xi},\vec{\eta}\rangle_\hilb := \sum_{j=1}^n\frac{1}{\pi^{n-1}}\int_{R_j}\xi_j(\vec{s})\bar{\eta_j(\vec{s})}\diff\lambda_{R_j}(\vec{s}).$$
With this choice, $(\hilb,\langle\cdot,\cdot\rangle_\hilb)$ is a Hilbert space.

Next, given $k \in \{1,\ldots,n\}$, define the linear relation $A_k$ via the equivalence
$$\{\vec{\xi},\vec{\eta}\} \in A_k \iff \xi_k = 0\wedge \forall j\in\{1,\ldots,n\}\setminus\{k\}:\,s_k\,\xi_j(\vec{s}) = \eta_j(\vec{s}) \quad\text{a.e. on}\;R_j.$$
Regarding the multi-valued part of $A_k$, we note first that
$$\{\vec{0},\vec{\eta}\} \in A_k \iff \forall j\in\{1,\ldots,n\}\setminus\{k\}:\; \eta_j(\vec{s}) = 0 \quad\text{a.e. on}\;R_j.$$
Hence,
$$\Mult(A_k) = \big\{(0,\ldots,0,\eta,0,\ldots,0) \in \hilb~\big|~\eta \in \Leb^2(R_k,\lambda_{R_k})\big\},$$
where the function $\eta$ lies at the $k$-th position. Therefore, $\Mult(A_k)$ equals $\{\vec{0}\} \subseteq \hilb$ when $b_k = 0$ and is isomorphic (as a Hilbert space) to $\Leb^2(\R^{n-1},\lambda_{\R^{n-1}})$ when $b_k \neq 0$.

Continuing, we need to check that each relation $A_k$ is self-adjoint. To that end, we have the following chain of equivalences:
$$\begin{array}{RCL}
\multicolumn{3}{L}{\{\vec{\xi},\vec{\eta}\} \in A_k^* \iff \forall \{\vec{f},\vec{g}\} \in A_k:\:\langle \vec{g},\vec{\xi} \rangle_\hilb = \langle \vec{f},\vec{\eta} \rangle_\hilb } \\[0.35cm] 
~ & \iff & \forall \{\vec{f},\vec{g}\} \in A_k: \\[0.1cm]
~ & ~ & \sum_{j=1}^n\frac{1}{\pi^{n-1}}\int_{R_j}g_j(\vec{s})\bar{\xi_j(\vec{s})}\diff\lambda_{R_j}(\vec{s}) = \sum_{j=1}^n\frac{1}{\pi^{n-1}}\int_{R_j}f_j(\vec{s})\bar{\eta_j(\vec{s})}\diff\lambda_{R_j}(\vec{s}). \\[0.55cm] 
~ & \iff & \forall \{\vec{f},\vec{g}\} \in A_k: \\[0.1cm]
~ & ~ & \sum_{\substack{j = 1 \\ j \neq k}}^n\int_{R_j}f_j(\vec{s})\bar{\big(s_k\,\xi_j(\vec{s}) - \eta_j(\vec{s})\big)}\diff\lambda_{R_j}(\vec{s}) + \int_{R_k}g_k(\vec{s})\bar{\xi_k(\vec{s})}\diff\lambda_{R_k}(\vec{s}) = 0. \\[0.75cm] 
~ & \iff & \{\vec{\xi},\vec{\eta}\} \in A_k.
\end{array}$$
Here, the claim that
$$\int_{R_j}f_j(\vec{s})\bar{\big(s_k\,\xi_j(\vec{s}) -
  \eta_j(\vec{s})\big)}\diff\lambda_{R_j}(\vec{s}) = 0 \Rightarrow
s_k\,\xi_j(\vec{s}) = \eta_j(\vec{s}) \quad\text{a.e. on}\;R_j$$
hinges on the fact that the projection of $\Dom(A_k)$ onto
$\Leb^2(R_j,\lambda_{R_j})$ is dense for $j\ne k$.  Indeed, it
certainly contains every compactly supported continuous function on
$R_j$.  This proves that the relation $A_{k}$ is self-adjoint.

Next, we determine the resolvent of $A_k$. To that end, we have, for any $\vec{z} \in \C^+$, the following chain of equivalences:
\begin{multline*}
    \{\vec{\xi},\vec{\eta}\} \in (A_k - z)^{-1} \iff \{\vec{\eta},\vec{\xi}\} \in (A_k - z) \iff \{\vec{\eta},\vec{\xi}+z\,\vec{\eta}\} \in A_k \\
    \iff \eta_k = 0 \:\wedge\:\forall j\in\{1,\ldots,n\}\setminus\{k\}:(s_k-z)\,\eta_j(\vec{s}) = \xi_j(\vec{s}) \quad\text{a.e. on}\;R_j.
\end{multline*}
As in Case 2, the resolvent relations $(A_k-z)^{-1}$ are everywhere-defined bounded linear operators. Explicitly,
for $\vec{\xi} \in \hilb$ it holds, for $j \in \{1,\ldots,n\}$, that
$$\big((A_k-z)^{-1}\vec{\xi}\big)_j(\vec{s}) = \left\{\begin{array}{RCL}
     0 & ; & j = k, \\
     (s_k - z)^{-1}\xi_j(\vec{s}) & ; & j \neq k.
\end{array}\right.$$

As in the previous case, let now $\vec{z},\vec{w} \in \C^{+n}$ be arbitrary and let them be fixed for the remainder of Case 3. Define $v_{\vec{w}}$ component-wise by setting, for each $j \in \{1,\ldots,n\}$,
$$(v_{\vec{w}})_{j} := \bigg(\vec{s} \mapsto \prod_{\substack{\ell = 1 \\ \ell \neq j}}^n(s_\ell - w_\ell)^{-1}\bigg).$$
As always, we note that the coordinates represented by the vector $\vec{s}$ change as the index $j$ varies and when given an index $\ell$ they equal the coordinates on $R_j$ as specified in the beginning of Case 3. Continuing, an analogous computation as in Case 2 shows that $(v_{\vec{w}})_{j} \in \Leb^2(R_j,\lambda_{R_j})$ for all $j \in \{1,\ldots,n\}$ and, hence, we indeed have $v_{\vec{w}} \in \hilb$.

Next, we are required to verify that expression
\eqref{eq:mobious_expression} is well-defined, this time with the
choices of $\hilb,A_k,v_{\vec{w}}$ as they have been made in Case
3. This readily follows by the same argument as in Case 2. Note,
additionally, for $j \in \{1,\ldots,n\}$, that
\begin{multline}
    \label{eq:action_2}
    \big((\Id + (z_{k} - w_{k})(A_{k} - z_{k})^{-1})v_{\vec{w}}\big)_j(\vec{s}) \\ = \left\{\begin{array}{RCL}
     (v_{\vec{w}})_j(\vec{s}) & ; & j = k, \\
     \left(1 + \frac{z_{k} - w_{k}}{s_{k} - z_{k}}\right)(v_{\vec{w}})_j(\vec{s}) & ; & j \neq k.
\end{array}\right.
\end{multline}

With this, our preparatory work is complete and by making repeated uses of equality \eqref{eq:action_2}, we calculate that
$$\begin{array}{RCL}
\multicolumn{3}{L}{\frac{1}{(2\,\I)^{n-1}}\prod_{k=1}^n(z_k-\bar{w_k})\cdot \bigg\langle\prod_{k=1}^n(\Id + (z_k - w_k)(A_k - z_k)^{-1})\,v_{\vec{w}},v_{\vec{w}}\bigg\rangle_\hilb} \\
    ~ & = & \frac{1}{(2\,\I)^{n-1}}\prod_{k=1}^n(z_k - \bar{w_k}) \\[0.45cm]
    ~ & ~ & \cdot\sum_{j=1}^n\bigg[\frac{1}{\pi^{n-1}}\int_{R_j}\prod_{\substack{\ell=1 \\ \ell \neq j}}^n\left(1+\frac{z_\ell - w_\ell}{s_\ell-z_\ell}\right)\prod_{\substack{\ell=1 \\ \ell \neq j}}^n\frac{1}{s_\ell-w_\ell}\prod_{\substack{\ell=1 \\ \ell \neq j}}^n\frac{1}{s_\ell-\bar{w_\ell}}\diff\lambda_{R_j}(\vec{s})\bigg] \\[0.45cm]
    ~ & = & \frac{1}{(2\,\I)^{n-1}}\prod_{k=1}^n(z_k - \bar{w_k})\\[0.45cm]
    ~ & ~ & \cdot\sum_{j = 1}^n\bigg[\frac{b_j}{\pi^{n-1}}\int_{\R^{n-1}}\prod_{\substack{\ell=1 \\ \ell \neq j}}^n\frac{1}{(s_\ell-z_\ell)(s_\ell-\bar{w_\ell})}\diff s_1 \,\ldots\, \diff s_{j-1}\,\diff s_{j+1} \,\ldots\, \diff s_n\bigg]  \\[0.45cm]
    ~ & = & \frac{1}{(2\,\I)^{n-1}}\prod_{k=1}^n(z_k - \bar{w_k})\cdot\sum_{j = 1}^n\bigg[b_j\prod_{\substack{\ell=1 \\ \ell \neq j}}^n\frac{2\,\I}{z_\ell-\bar{w_\ell}}\bigg] = \sum_{j = 1}^nb_j(z_j - \bar{w_j})  = (**2).
\end{array}$$
For the first equality on the last line, see
Example \ref{ex:Poisson_type_Lebsgue}. Using
Theorem \ref{thm:positive_semi_decomposition} while noting that we are
in the case when $\mu = 0$, we arrive at
$$(**2) = q(\vec{z}) - \bar{q(\vec{w})},$$
thereby leading to equality \eqref{eq:opRep_Nvar}.  Again, we note
that the above computation also establishes \eqref{eq:admissible-v_w}.
This concludes Case 3.

\textit{Case 4:} $\vec{b} \neq \vec{0}, \mu \neq 0$. This case is resolved by taking a direct sum of all the objects (\ie the Hilbert spaces, the relations, the vector $v_{\vec{w}}$) from Cases 2 and 3. The result follows immediately, concluding Case 4.

This concludes the proof, as we have successfully shown that every Herglotz-Nevanlinna function admits a representation of the desired form.
\endproof

\begin{remark}
Ending up in Cases 1 or 2, when all the linear relations $A_1,\ldots,A_n$ can be chosen as usual operators, can be described by the properties of the function $q$. By \cite[Cor. 4.6.(vi)]{LugerNedic2019}, it holds that
$$b_j = \lim\limits_{z_j \ntto \infty}\frac{q(\vec{z})}{z_j}$$
for all $j = 1,\ldots,n$. This gives an analogous dichotomy as that of Theorem \ref{thm:all_rep_1var}.c, 
\end{remark}

Next, we present a description of Poisson-type functions that was implicitly established in Case 2 of the above proof.

\begin{prop}\label{prop:poisson_rep}
  Let $D$ be a Poisson type function, as in Definition
  \ref{def:poisson_type}.  Then, there exists a Hilbert space
  $(\hilb,\langle\cdot,\cdot\rangle_\hilb)$ and $n$ self-adjoint
  linear operators $A_1,\ldots,A_n$ in $\hilb$ such that to every
  $\vec w\in\C^{+n}$ corresponds an element $v_{\vec{w}} \in \hilb$
  for which
$$D(\vec{z},\vec{w}) = \bigg\langle\prod_{k=1}^n(\Id + (z_k - w_k)(A_k
- z_k)^{-1})\,v_{\vec{w}},v_{\vec{w}}\bigg\rangle_\hilb$$
for all $\vec z\in\C^{+n}$.
\end{prop}

\subsection{What kind of function is defined by the representation}

We need to now verify that given a Hilbert space and $n$ self-adjoint linear relations, the representation \eqref{eq:opRep_Nvar} does indeed define a Herglotz-Nevanlinna function.

\begin{prop}
\label{prop:op_to_hn}
  Let $(\hilb,\langle\cdot,\cdot\rangle_\hilb)$ be a complex Hilbert
  space, let $A_1,\ldots,A_n$ be $n$ self-adjoint linear relations on
  $\hilb$, and let $q\colon\C^{+n} \to \C$ be a function. Assume that
  for to every $\vec{w} \in \C^{+n}$ corresponds an element
  $v_{\vec{w}} \in \hilb$ for which representation
  \eqref{eq:opRep_Nvar} holds. Then, $q$ is necessarily a
  Herglotz-Nevanlinna function.
\end{prop}

\proof
Let $q\colon\C^{+n} \to \C$ be a function, let $(\hilb,\langle\cdot,\cdot\rangle_\hilb)$ be a Hilbert space, let $A_1,\ldots,A_n$ be as $n$ self-adjoint linear relations in $\hilb$, let $\vec{z},\vec{w} \in \C^{+n}$ be arbitrary and let $v_{\vec{w}} \in\hilb$ be an element of $\hilb$ depending on $\vec{w}$ but not on $\vec{z}$.

Assume now that equality \eqref{eq:opRep_Nvar} holds under these stipulations. Then, we claim that the function $q$ is necessarily holomorphic on $\C^{+n}$ and that $\Im[q(\vec{z})] \geq 0$ for all $\vec{z} \in \C^{+n}$ (\ie we claim that $q$ is a Herglotz-Nevanlinna function). To that end, we note by Lemmas \ref{lem:resolvent_identity} and \ref{lem:self_adjoint} that $q$ is holomorphic on $\C^{+n}$. On the other hand, it also readily follows that $\Im[q(\vec{z})] \geq 0$ for all $\vec{z} \in \C^{+n}$ as choosing $\vec{z} = \vec{w}$ in equality \eqref{eq:opRep_Nvar} yields
$$2\,\I\,\Im[q(\vec{z})] = 2\,\I\,\prod_{k=1}^n\Im[z_k]\cdot\|v_{\vec{z}}\|_{\hilb}^2.$$
This finishes the proof.
\endproof

We emphasize that for most choices of linear relations
$A_1,\ldots,A_n$ and maps $\vec w\to v_{\vec w}$, there exists no
function $q$ which satisfies \eqref{eq:opRep_Nvar}. The content of
Proposition \ref{prop:op_to_hn} is that if $A_1,\ldots,A_n$ and
$\vec w\to v_{\vec w}$ satisfy the additional condition that the
second term of the right side of \eqref{eq:opRep_Nvar} is of the form
$q(\vec z)-\overline{q(\vec w)}$ for some function $q$, then this $q$
is a Herglotz-Nevanlinna function. This phenomena is related to the fact that a Poisson-type function
$D(\vec z, \vec w)$ given by a measure $\mu$ is of the form
$q(\vec z)-\overline{q(\vec w)}$ if and only if $\mu$ is a Nevanlinna
measure, \cf\cite[Thm. 4.1]{Nedic2021b}. If $\mu$ is not a Nevanlinna
measure and $A_1,\ldots,A_n$ and $\vec w\to v_{\vec w}$ are
constructed from $D$ as in Proposition \ref{prop:poisson_rep}, then no $q$ satisfies \eqref{eq:opRep_Nvar}.

The following version of Theorem \ref{thm:opRep_Nvar} gives an
explicit condition on $A_1,\ldots,A_n$ and $\vec w\to v_{\vec w}$
under which a function $q$ satisfying \eqref{eq:opRep_Nvar} can be
constructed.

\begin{thm}
\label{thm:opRep_alt}
  If $q$ is a Herglotz-Nevanlinna function, there exists a Hilbert space   $(\hilb,\langle\cdot,\cdot\rangle_\hilb)$ and self-adjoint relations
  $A_1,\ldots,A_n$ on $\hilb$ such that to every
  $\vec{\alpha}\in\C^{+n}$ corresponds an element $v\in\hilb$ with
\begin{equation}\label{e.classbetter}
q(\vec{z}) = \overline{q(\vec{\alpha})} + S(\vec{z},\vec{\alpha})\quad\text{and}\quad S(\vec{z},\vec{w})+S(\vec{\alpha},\vec{\alpha})=S(\vec{z},\vec{\alpha})+S(\vec{\alpha},\vec{w}),
\end{equation}
where
\begin{equation*}
S(\vec{z},\vec{w})=p(\vec{z}-\overline{\vec{w}})L(\vec{w}, \vec{z}),\quad
p(\vec{z})=2\,\I\prod_{k=1}^n\frac{z_k}{2\,\I},\quad L(\vec{w},\vec{z})=\langle\varphi_{\vec\alpha}(\vec{z})v,\varphi_{\vec\alpha}(\vec{w})v\rangle_\hilb,
\end{equation*}
\begin{equation*}
\varphi_{\vec\alpha}(\vec{z})=\prod_{k=1}^n(\Id+(z_k-\alpha_k)(A_k-z_k)^{-1}).
\end{equation*}
Conversely, given any self-adjoint relations $A_1,\ldots,A_n$ on a
Hilbert space $\hilb$, any $\vec{\alpha}\in \C^{+n}$, and any element
$v\in\hilb$, the function $q$, which the first half of
\eqref{e.classbetter} defines up to an additive real constant, is
Herglotz-Nevanlinna, provided that the second half of
\eqref{e.classbetter} is satisfied.
\end{thm}

Note that because $\overline{S(\vec z, \vec w)} = -S(\vec w, \vec z)$,
the pair of equalities \eqref{e.classbetter} is equivalent to the
statement that $q(\vec z)-\overline{q(\vec w)}=S(\vec z, \vec w)$ for
all $\vec z$ and $\vec w$. Additionally, when $n=1$, the second half of \eqref{e.classbetter} reduces to the resolvent identity. Thus, it holds for any choice of $\hilb$, $A_1$,
$\alpha$ and $v$.

\proof
Let $q$ be Herglotz-Nevanlinna and $\vec\alpha\in\C^{+n}$. By
Proposition \ref{prop:hn_to_op} there exist $\hilb$, $A_1,\ldots,A_n$,
and $\vec w\to v_{\vec w}$ such that
\begin{equation*}
  q(\vec z) = \overline{q(\vec w)} + p(\vec z-\overline{\vec
    w})\langle{}\varphi_{\vec w}(\vec z)v_{\vec w}, v_{\vec w}\rangle{}
\end{equation*}
and
\begin{equation*}
  v_{\vec z} = \varphi_{\vec w}(\vec z)v_{\vec w}
\end{equation*}
for all $\vec z, \vec w \in \C^{+n}$.  If $v=v_{\vec\alpha}$ it
follows that
$\varphi_{\vec w}(\vec z)v_{\vec w}=\varphi_{\vec\alpha}(\vec z)v$,
$v_{\vec w}=\varphi_{\vec\alpha}(\vec w)v$, and therefore
\begin{equation*}
  q(\vec z)-\overline{q(\vec w)}=p(\vec z-\overline{\vec
    w})\langle{}\varphi_{\vec\alpha}(\vec
  z)v,\varphi_{\vec\alpha}(\vec w)v\rangle{}_\hilb=S(\vec z, \vec w).
\end{equation*}
Conversely, if \eqref{e.classbetter} holds, then
\begin{equation*}
  \frac{q(\vec{z})-\overline{q(\vec{w})}}{p(\vec{z}-\overline{\vec{w}})}=\frac{S(\vec
  z, \vec w)}{p(\vec{z}-\overline{\vec{w}})}=\langle\varphi_{\vec\alpha}(\vec{z})v,\varphi_{\vec\alpha}(\vec{w})v\rangle_\hilb,
\end{equation*}
which is positive semi-definite. This finishes the proof.
\endproof

We conclude this section with an example showing how the choice of representing parameters $\hilb$ and $A_k$ may restrict the type of Herglotz-Nevanlinna function these parameters might represent.

\begin{example}\label{ex:mapping}

Starting with a Herglotz-Nevanlinna function $q$, the constructive nature of the proof of Proposition \ref{prop:hn_to_op} gives a readily available option of a Hilbert space and corresponding self-adjoint relations. Conversely, as noted after the proof of Proposition \ref{prop:op_to_hn}, an arbitrary choice of $\hilb$ and $A_j$ and $\vec{w} \mapsto v_{\vec{w}}$ might only be able to produce a real-constant Herglotz-Nevanlinna function. Equivalently, in terms of Theorem \ref{thm:opRep_alt}, the function $S$ would be identically zero. To exemplify this phenomena, we consider the following.

Let $\hilb = \C^1$ with the usual skew-symmetric Euclidean inner product, and choose $A_1 = \ldots = A_n= 0$. Suppose that the function
\begin{equation*}
  S(\vec{z}, \vec{w}) = p(\vec{z}-\bar{\vec{w}})\langle\varphi_{\vec{w}}(\vec{z})v_{\vec{w}}, v_{\vec{w}}\rangle
\end{equation*}
is of the form $q(\vec{z}) - \bar{q(\vec{w})}$ for some function $q$. As in the proof of
Theorem \ref{thm:opRep_alt}, the following two conditions must then be satisfied by the function $S$,
\begin{equation*}
  \bar{S(\vec{z}, \vec{w})} = -S(\vec{w}, \vec{z}),\qquad S(\vec{z}, \vec{w}) + S(\vec{\alpha}, \vec{\alpha}) = S(\vec{z}, \vec{\alpha}) + S(\vec{\alpha}, \vec{w}).
\end{equation*}
Since
\begin{equation*}
  p(\vec{z}-\bar{\vec{w}}) = 2\,\I\,\prod_{k=1}^n\frac{z_k-\bar{w_k}}{2\,\I},\qquad \langle\varphi_{\vec{w}}(\vec{z})v_{\vec{w}},v_{\vec{w}}\rangle = |v_{\vec{w}}|^2\prod_{k=1}^n\frac{w_k}{z_k},
\end{equation*}
the first condition implies that
\begin{equation*}
  |v_{\vec{w}}|^2\prod_{k=1}^n|w_k|^2 = |v_{\vec{z}}|^2\prod_{k=1}^n|z_k|^2.
\end{equation*}
Therefore, $|v_{\vec{w}}|^2 = C\prod_{k=1}^n |w_k|^{-2}$ for some constant $C$, yielding
\begin{equation*}
  S(z, w) = 2\,\I\,C\prod_{k=1}^n\frac{z_k-\bar{w_k}}{2\,\I\,z_k\,\bar{w_k}}.
\end{equation*}
Choosing $z_k = w_k = \I\,y$ and $\alpha_k = \I$ for all $k=1,\ldots,n$, the second condition now
reads
\begin{equation*}
  \frac{2\,\I\,C}{y^n} + 2\,\I\,C = 4\,\I\,C\left(\frac{1}{2}+\frac{1}{2y}\right)^n.
\end{equation*}
For this to be satisfied for all $y > 0$, it must hold that either $C = 0$ (and therefore $S \equiv 0$) or $n=1$. \hfill $\lozenge$
\end{example}

\section{Comparison to other types of representations}
\label{sec:comparisons}

In general, one can ask how the representation theorems \ref{thm:opRep_Nvar} and \ref{thm:opRep_alt} relate either to other classes of holomorphic functions, or to other methods that can be used to construct similar representations.

\subsection{Operator representation of Loewner functions}
\label{sec:loewner_rep}

For Herglotz-Nevan\-linna functions that fall into the subclass of \emph{Loewner functions}, different operator representations were previously established in \cite{AglerEtal2016}. Hence, in this section, we briefly review the main representation for this class and contrast it to the ones presented in this paper. We begin by recalling the following definition, \cf \cite[pg. 3003]{AglerEtal2016}.

\begin{define}
    \label{def:loewner_function}
    A holomorphic function $q\colon \C^{+n} \to \C$ is called a \emph{Loewner function} (on $\C^{+n}$) if there exist $n$ positive semi-definite functions $F_1,\ldots,F_n$ on $\C^{+n} \times \C^{+n}$, such that the equality
    \begin{equation}
    \label{eq:lowner_functions}
        q(\vec{z}) - \bar{q(\vec{w})} = \sum_{\ell = 1}^n(z_\ell - \bar{w_\ell})F_\ell(\vec{z},\vec{w}).
    \end{equation}
    holds for all $\vec{z},\vec{w} \in \C^{+n}$.
\end{define}

In dimensions $n=1,2$, the Loewner functions coincide with the Herglotz-Nevanlinna functions, while in dimensions $n \geq 3$ they form a strict subclass of the Herglotz-Nevanlinna functions \cite[Sec. 1]{AglerEtal2016}. The main representation theorem for Loewner functions is the following \cite[Thm. 1.9]{AglerEtal2016}.

\begin{thm}
\label{thm:loewner_rep}
A function $q\colon \C^{+n} \to \C$ is a Loewner function if and only if there exist two Hilbert spaces     $(\hilb_1,\langle\cdot,\cdot\rangle_{\hilb_1})$ and $(\hilb_2,\langle\cdot,\cdot\rangle_{\hilb_2})$, a densely defined self-adjoint operator $A$ on $\hilb_2$, a vector $v \in \hilb_1 \oplus \hilb_2$ and $n$ pair-wise orthogonal projections $P_1,\ldots,P_n$ acting on $\hilb_1 \oplus \hilb_2$ that sum up to $\Id_{\hilb_1 \oplus \hilb_2}$ such that it holds, for $\vec{z} \in \C^{+n}$, that 
$$q(\vec{z}) = a + \langle M(\vec{z})v,v \rangle_{\hilb_1 \oplus \hilb_2}$$
where $a \in \R$ is a constant and
\begin{multline*}
    M(\vec{z}) := \left(-\I\matfour{\Id_{\hilb_1}}{0}{0}{A} + \matfour{0}{0}{0}{\Id_{\hilb_2}}\right) \\ 
    \cdot\left(\matfour{\Id_{\hilb_1}}{0}{0}{A} - (z_1P_1 + \ldots + z_nP_n)\matfour{0}{0}{0}{\Id_{\hilb_2}}\right)^{-1} \\
    \cdot \left((z_1P_1 + \ldots + z_nP_n)\matfour{\Id_{\hilb_1}}{0}{0}{A} + \matfour{0}{0}{0}{\Id_{\hilb_2}}\right) \\
    \cdot \left(-\I\matfour{\Id_{\hilb_1}}{0}{0}{A} + \matfour{0}{0}{0}{\Id_{\hilb_2}}\right)^{-1}.
\end{multline*}
\end{thm}

For Loewner functions which additionally satisfy specific growth conditions, alternative representations are known, see \cite[Thms. 1.6, 1.7 and 1.8]{AglerEtal2016}. When comparing the statements of Theorem \ref{thm:opRep_Nvar} (and Theorem \ref{thm:all_rep_1var}.c) to Theorem \ref{thm:loewner_rep}, and the objects appearing in them, the following distinctions are observed.

In general, Theorem \ref{thm:opRep_Nvar} covers a larger class of functions than Theorem \ref{thm:loewner_rep}. In the cases $n=1,2$ when the two classes of Herglotz-Nevanlinna functions and Loewner function are equal, there is at present no known correspondence between the parameters of the two theorems. Taking Theorem \ref{thm:all_rep_1var}.c as a starting point, it is Theorem \ref{thm:opRep_Nvar} that provides its most direct possible generalization in terms of which objects (\ie one Hilbert space, and $n$ self-adjoint linear relations) are used in the representation. Moreover, Theorem \ref{thm:opRep_Nvar}, through the proof of Proposition \ref{prop:hn_to_op}, provides an explicit example of which Hilbert space and relations may be used to represent a given function, which are in turn the most direct possible analogue to the corresponding construction from Theorem \ref{thm:all_rep_1var}.c. With regards to the different in the parameters used in the representations formulas of two theorems, Theorem \ref{thm:opRep_Nvar} uses a product of resolvents of different self-adjoint linear relations, but all of the relations are defined on the same Hilbert space. On the other hand, Theorem \ref{thm:loewner_rep} uses two Hilbert spaces and only one self-adjoint operator along with $n$ pair-wise orthogonal projections. Finally, with regards to the setting of the proof, the entirety of the proof of Theorem \ref{thm:opRep_Nvar} is done in the several variable setting, while the proof of Theorem \ref{thm:loewner_rep} relies on results and reduction from the one variable setting.

\subsection{Representations using reproducing kernel Hilbert spaces}
\label{sec:rep_kernel}

Using the language of reproducing kernel Hilbert spaces, Theorem \ref{thm:opRep_alt} can be reformulated as follows.

\begin{thm}
\label{thm:rkhs}
Suppose the kernel $L$ of a reproducing kernel Hilbert space of holomorphic functions on $\C^{+n}$ satisfies 
\begin{equation}\label{e.kernelrel}
S(\vec{z},\vec{w})+S(\vec{\alpha},\vec{\alpha})=S(\vec{z},\vec{\alpha})+S(\vec{\alpha},\vec{w})\qquad\,(\vec{z},\vec{w}\in\C^{+n})
\end{equation}
for some $\vec{\alpha}\in\C^{+n}$, where $S(\vec{z},\vec{w}) = p(\vec{z}-\overline{\vec{w}})L(\vec{w}, \vec{z})$ and $p$ is as in Theorem \ref{thm:opRep_alt}. Then, there exists a Herglotz-Nevanlinna function $q$ such that
\begin{equation}\label{e.kernelrel2}
L(\vec{w},\vec{z})=\frac{q(\vec{z})-\overline{q(\vec{w})}}{p(\vec{z}-\overline{\vec{w}})}.
\end{equation}
Conversely, for any Herglotz-Nevanlinna function $q$, the kernel $L$ is positive semi-definite and satisfies formula \eqref{e.kernelrel} for all $\vec{\alpha}\in\C^{+n}$. Moreover, if $\hilb=\Leb^2(\R^n,\nu) \oplus \Leb^2(\R^n,\nu_1) \oplus \ldots \oplus
\Leb^2(\R^n,\nu_n)$, where $\diff\nu=\diff\mu/\pi^n$, the vector
$\vec{b}$ is from representation \eqref{eq:intRep_Nvar}, and
\begin{equation*}
\int_{\R^n}f(\vec{s})\diff\nu_j(\vec{s})=\frac{b_j}{\pi^{n-1}}\int_{\R^{n-1}}f(s_1,\ldots,s_{j-1},0,s_j,\ldots,s_{n-1})\diff\vec{s},
\end{equation*}
and if the $\Leb^2(\R^n,\nu)$- and $\Leb^2(\R^n,\nu_j)$-components of $h_{\vec{z}}\in \hilb$ are
\begin{equation*}
\prod_{k=1}^n(t_k-\overline{z_k})^{-1}\quad\text{and}\quad(t_j-\overline{z_j})\prod_{k=1}^n(t_k-\overline{z_k})^{-1},
\end{equation*}
respectively, then $(Uf)(\vec{z})=\langle f,h_{\vec z} \rangle_\hilb$ is an isometric isomorphism of the closed span of $\{h_{\vec{z}}\colon \vec{z}\in\C^{+n}\}$ onto the reproducing kernel space corresponding to $L$.
\end{thm}
\proof If the kernel $L$ satisfies \eqref{e.kernelrel}, then
\eqref{e.kernelrel2} holds with
\begin{equation*}
  q(\vec z) = -\tfrac{1}{2}S(\vec\alpha, \vec\alpha) + S(\vec z, \vec\alpha).
\end{equation*}
Because $L(\vec w, \vec z)$ is holomorphic in $\vec z$, this $q$ is holomorphic. Since $L$ is positive semi-definite,
\eqref{e.kernelrel2} shows that $q$ is Herglotz-Nevanlinna.

For the converse, assume $q$ is Herglotz-Nevanlinna and let $X$ be the range of $U$ equipped with the inner product
$\langle Uf, Ug\rangle_X = \langle f, g\rangle_\hilb$.  It is clear
that $X$ is a reproducing kernel Hilbert space, with kernel
$K(\vec w, \vec z) = \langle h_{\vec w}, h_{\vec z}\rangle_\hilb$.  In the proof of Proposition \ref{prop:hn_to_op}, we showed that
$\langle h_{\vec w}, h_{\vec z}\rangle_\hilb = L(\vec w, \vec z)$.
Since the reproducing kernel space corresponding to a given kernel is
unique, the theorem follows. 
\endproof

In one variable, the Stieltjes formula shows that
$\{(t-z)^{-1}\colon z\in\cut\}$ spans a dense subspace of
$\Leb^2(\R,\nu)$. By induction, the same is true in several variables,
if the $\vec{z}$ in $\prod_{k=1}^n(t_k-z_k)^{-1}$ ranges over
$\cutN$. Thus, is would be desirable to obtain an analogous statement
where $q$ is extended to $\cutN$ in a suitable way. However,
considering \eg \cite[Sec. 4.3]{Nedic2021b}, such an investigation
lies outside the scope of the present paper. One would also like to
find a more interesting way of stating \eqref{e.kernelrel}, more akin
to \eg \cite[Thms. 5, 6]{deBranges1968}.

On a related note, one could ask whether the representation in Theorem \ref{thm:opRep_alt} is minimal in any suitable sense, and to what extent a minimal representation is unique.

\section*{Acknowledgements}

The authors would like to thank Annemarie Luger for many enthusiastic discussions on the subject, as well as careful reading of the manuscript and many helpful comments and suggestions.

\end{document}